\documentclass[11pt]{amsart}
\usepackage{amscd}
\usepackage{amssymb}
\usepackage[margin=1in,includeheadfoot]{geometry}
\usepackage{cite}
\usepackage [latin1]{inputenc}
\usepackage{tabularx,booktabs,tikz}
\usepackage{caption}
\usepackage{amsmath}
\usepackage{amsfonts}

\usepackage{amscd}
\usepackage{amsthm}
\usepackage{amssymb} \usepackage{latexsym}
\usepackage{eufrak}
\usepackage{euscript}
\usepackage{epsfig}
\usepackage{graphics}
\usepackage{array}
\usepackage{enumerate}
\usepackage{dsfont}
\usepackage{color}
\usepackage{wasysym}
\usepackage{hyperref}
\usepackage{pdfsync}

\def\beq{\begin{equation}}
\def\eeq{\end{equation}}
\def\ba{\begin{array}}
\def\ea{\end{array}}

\def\R{\mathbb R}

\newtheorem{thm}{Theorem}[section]
\newtheorem{lm}[thm]{Lemma}
\newtheorem{prop}[thm]{Proposition}

\theoremstyle{definition}
\newtheorem{rem}[thm]{Remark}

\theoremstyle{remark}

\begin{document}
\pagestyle{plain}
\title{The existence and convergence of solutions for the nonlinear Choquard equations on groups of polynomial growth }

\author{Ruowei Li}
\email{rwli19@fudan.edu.cn}
\address{Ruowei Li: School of Mathematical Sciences, Fudan University, Shanghai 200433,
People's Republic of China; Shanghai Center for Mathematical Sciences, Fudan University, Shanghai 200433, People's Republic of China}

\author{Lidan Wang}
\email{wanglidan@fudan.edu.cn}
\address{Lidan Wang: School of Mathematical Sciences, Fudan University, Shanghai 200433, People's Republic of China}


\begin{abstract}
In this paper, we study the nonlinear Choquard equation
\begin{eqnarray*}
\Delta^{2}u-\Delta u+(1+\lambda a(x))u=(R_{\alpha}\ast|u|^{p})|u|^{p-2}u
\end{eqnarray*}
on a Cayley graph of a discrete group of polynomial growth with the homogeneous dimension $N\geq 2$, where $\alpha\in(0,N),\,p>\frac{N+\alpha}{N},\,\lambda$ is a positive parameter and $R_\alpha$ stands for the Green's function of the discrete fractional Laplacian, which has same asymptotics as the Riesz potential. Under some assumptions
on $a(x)$, we establish the existence and asymptotic behavior of ground state solutions for the nonlinear Choquard equation by the method of Nehari manifold.
\end{abstract}

 \maketitle



\section{Introduction}
The nonlinear Choquard equation
\begin{equation}\label{0}
-\Delta u+V(x)u=(I_{\alpha}\ast|u|^{p})|u|^{p-2}u,\quad x\in \mathbb{R}^N,
\end{equation}
where $I_\alpha=\frac{\Gamma(\frac{N-\alpha}{2})}{\Gamma(\frac{\alpha}{2})\pi^{\frac{N}{2}}2^\alpha|x|^{N-\alpha}}$ is the Riesz potential, arises in various fields of mathematical physics, such as the description of the quantum
theory of a polaron at rest by S. Pekar \cite{P1} and the modelling of
an electron trapped in its own hole in the work of P. Choquard in 1976. It was
also treated as a certain approximation to Hartree-Fock theory of one-component
plasma \cite{L0}. Sometimes the equation (\ref{0}) was also known as the Schr\"{o}dinger-Newton
equation \cite{P2}, since the convolution part might be treated as a coupling with a
Newton equation.

In the last decades, a great deal of mathematical efforts has been devoted to the study of existence, multiplicity
and properties of solutions to the nonlinear Choquard equation (\ref{0}). If the
equation (\ref{0}) is equipped with a positive constant potential $V$, for $N=3,\,\alpha=2$ and $p=2$, Lieb \cite{L0} proved the existence and uniqueness, up to translations, of the ground state solution by the rearrangement technique. Later Lions \cite{L1} showed the existence of radially symmetric solutions by variational methods. For $N\geq 3$, Moroz and Van Schaftingen \cite{MV2} proved the existence of a positive ground state solution for the optimal parameter range of  $\frac{N+\alpha}{N}<p<\frac{N+\alpha}{N-2}$, and showed the regularity, positivity, radial symmetry and decay property at infinity. If the equation (\ref{0}) is equipped with the deepening potential well $V(x)=1+\lambda a(x)$, where $a(x)$ satisfies the assumptions first introduced by Bartsch and Wang \cite{BW},  then Alves et al. \cite{ANY} proved the existence and multiplicity of multi-bump shaped solutions for large $\lambda$. In \cite{L4}, L\"{u} proved the existence of ground state solutions and the sequence of solutions converges strongly to a ground state solution for the problem in a bounded domain if $\lambda$ is large enough. For more works about the existence and concentration behaviour of solutions to the Choquard equations, we refer readers to \cite{ANY1,ANY2,MV4}. In addition, there are intensive studies of the Choquard equations, see papers \cite{CCS,GV,L4,MZ,MV1,MV2,MV3}.

Nowadays, people paid attention to the analysis on discrete spaces, especially
for the nonlinear elliptic equations, see for examples \cite{CMW,GLY3,GJ,GLY2,HL,HLW,HSZ,HX,LZY,M,ZZ}. It is worth noting that Zhang and Zhao \cite{ZZ} proved the existence and concentration behaviour of ground state solutions for a class of semilinear elliptic equations equipped with the potential $V(x)=1+\lambda a(x)$ with $a(x)\rightarrow+\infty$ as $d(x,x_0)\rightarrow+\infty$. Later, the authors \cite{HSZ} generalized their results to biharmonic equations defined on graphs. As far as we know, there are no such results for the nonlinear Choquard equations on graphs. Hence, in this article, we study a class of Choquard equations on graphs.

Let $(G, S)$ be a Cayley graph of a group $G$ with a finite symmetric generating set $S$, i.e. $S=S^{-1}.$ There is a natural metric on $(G, S)$ called the word metric, denoted by $d^S.$ Let $B^{S}_{r}(a)=\{x\in G: d^{S}(x,a)\leq r\}$ be the closed ball of radius $r$ centered at $a\in G$ and denote $|B^{S}_{r}(a)|=\sharp B^{S}_{r}(a)$ as the volume of the set $B^{S}_{r}(a)$. Let $e$ be the unit element of $G$, the volume $\beta^{S}(r):=|B^{S}_{r}(e)|$ of $B^{S}_{r}(e)$ is called of growth function, see \cite{G1,G2,M1,M2,W1}.
A group $G$ is called of polynomial growth, or of polynomial volume growth, if there exists a
finite generating set $S$ such that $\beta^{S}(r)\leq Cr^A$ for any $r\geq1$ and some $A >0 $. This definition is independent of the choice of the generating set $S$ since the metrics $d^S$ and $d^{S_1}$ are bi-Lipschitz equivalent for the different finite generating sets $S$ and $S_1$. By Gromov's theorem and Bass' volume
growth estimate of nilpotent groups \cite{B}, for any group $G$ of polynomial growth, there
are constants $C_1$ and $C_2$ depending on $S$ and $N\in \mathbb{N}$ such that for any $r\geq 1$,
$$C_1 r^N\leq\beta^{S}(r)\leq C_2 r^N, $$
where the integer $N$ is called the homogeneous dimension or the growth degree of
$G$.

In this paper, we consider the Cayley graph $(G, S)$ of a group of polynomial
growth with the homogeneous dimension $N\geq 2$. In particular, $\mathbb{Z}^N$ is a Cayley
graph of a free abelian group. Let $\Omega$ be a subset of $G$, we denote by $C(\Omega)$ the space of functions on $\Omega$. The support of $u\in C(\Omega)$ is
defined as $\text{supp}(u):=\{x \in \Omega : u(x)\neq 0\}$. Moreover, we denote by the $\ell^p(\Omega)$ the space of $\ell^p-$summable functions on $\Omega$. For convenience, for any $u\in C(\Omega)$, we always write
$
\int_{\Omega}u\,d\mu:=\sum\limits_{x\in \Omega}u(x),$ where $\mu$ is the counting measure in $\Omega$.

The Hardy-Littlewood-Sobolev (HLS for abbreviation) inequality plays a key role in the study of the Choquard equations. It is worth mentioning that Huang et al. \cite{HLY} have proved the discrete HLS inequality with the Riesz potential $I_\alpha$ on $\mathbb{Z}^N$, that is, for $u\in \ell^r(\mathbb{Z}^N),\,v\in \ell^s(\mathbb{Z}^N)$,
\begin{equation}\label{bz}
\int_{\mathbb{Z}^N}(I_\alpha\ast u)(x)v(x)\,d\mu:=\underset {x,y\in\mathbb{Z}^N, x\neq y}{\sum}\frac{u(y)v(x)}{|x-y|^{N-\alpha}}\leq C_{r,s,\alpha,N}\|u\|_r\|v\|_s,
\end{equation}
where $0<\alpha<N,\,1<r,s<+\infty$ and $\frac{1}{r}+\frac{1}{s}+\frac{N-\alpha}{N}=2$. Since the Riesz potential $I_\alpha$ in the equation (\ref{0}) is exactly the Green's function of the fractional Laplacian on $\mathbb{R}^N$, it is natural to consider the discrete fractional Laplacian and its
Green's function $R_\alpha$ on $G$ (see \cite{MC} for $\mathbb{Z}^N$),
\begin{equation}\label{xm}
R_{\alpha}(x,y)=\frac{1}{|\Gamma(\frac{\alpha}{2})|}\int_{0}^{+\infty}k_t(x,y)t^{-1+\frac{\alpha}{2}}\,dt,\quad x,y\in G.
\end{equation}
Since $G$ with the homogenous dimension
$N$ satisfies the weak Poincar\'{e} inequality, the heat kernel $k_t(x, y)$ has the Gaussian heat kernel bounds \cite{B1}. Hence the Green's function $R_\alpha$ of the fractional Laplacian has the asymptotic behavior $R_{\alpha}(x,y)\simeq (d^{S}(x,y))^{\alpha-N}$ with $N>\alpha.$ By the Young's inequality for weak type spaces \cite [Theorem~1.4.25]{G}, we get the discrete HLS inequality on $G$ in Lemma \ref{lm1}, that is, for $u\in \ell^r(G),\,v\in \ell^s(G),$
\begin{equation}\label{bp}
\int_{G}(R_\alpha\ast u)(x)v(x)\,d\mu:=\underset {x,y\in G,x\neq y}{\sum}R_\alpha(x,y)u(y)v(x)\leq C_{r,s,\alpha,N}\|u\|_r\|v\|_s,
\end{equation}
where $0<\alpha<N,\,1<r,s<+\infty$ and $\frac{1}{r}+\frac{1}{s}+\frac{N-\alpha}{N}=2$.

In this paper, we study the nonlinear Choquard equation
\begin{equation}\label{1}
\Delta^{2}u-\Delta u+(1+\lambda a(x))u=(R_{\alpha}\ast|u|^{p})|u|^{p-2}u
\end{equation}
on $G$ with the homogenous dimension
$N\geq2$, where $\alpha\in(0,N),\,p>\frac{N+\alpha}{N},\,\lambda>0$ is a parameter and $R_\alpha$ is the Green's function of the discrete fractional Laplacian as defined above (\ref{xm}). Here $\Delta$ is the Laplace operator of $u\in C(G)$ defined as $\Delta u(x)=\underset {y\sim x}{\sum}(u(y)-u(x))$. And $\Delta^2$ is the biharmonic operator of $u\in C(G)$ defined in the distributional sense as, for any $\phi\in C_{c}(G)=\{u\in C(G): \text{supp}(u)\subset G~ \text{is~of~finite~cardinality}\}$,
\begin{eqnarray*}
\int_{G}(\Delta^{2}u)\phi\,d\mu=\int_{G}\Delta u\Delta\phi\,d\mu.
\end{eqnarray*}

We assume that the nonnegative potential $a(x)$ satisfies:
 \begin{itemize}
\item[($A_1$)]  The potential well $\Omega=\{x\in G:a(x)=0\}$ is a nonempty, connected and bounded domain in $G$.
\item[($A_2$)] There exists $M>0$ such that the set $\{x\in G: a(x)\leq M\}$ is finite and nonempty.

\end{itemize}
The assumptions on $a(x)$ were first introduced by Bartsch and Wang \cite{BW} in the study of a nonlinear Schr\"{o}dinger equation, which are weaker than those in \cite{HSZ,ZZ}.

Let $W^{2,2}(G)$ be the completion of $C_c(G)$ with respect to the norm
\begin{eqnarray*}
\|u\|_{W^{2,2}}=(\int_{G}(|\Delta u|^2+|\nabla u|^2+u^2)\,d\mu)^{\frac{1}{2}}.
\end{eqnarray*}

For any $\lambda>0$, we introduce a new subspace $E_{\lambda}=\{u\in W^{2,2}(G): \int_{G}\lambda a(x)u^2\,d\mu<+\infty\}$
equipped with the norm
\begin{eqnarray*}
\|u\|_{E_{\lambda}}=(\int_{G}(|\Delta u|^2+|\nabla u|^2+(1+\lambda a)u^2)\,d\mu)^{\frac{1}{2}}.
\end{eqnarray*}

The energy functional $J_{\lambda}(u): E_{\lambda}\rightarrow\R$ associated to the equation (\ref{1}) is given by
\begin{eqnarray*}
J_{\lambda}(u)=\frac{1}{2}\int_{G}(|\Delta u|^2+|\nabla u|^{2}+(1+\lambda a)u^2)\,d\mu-\frac{1}{2p}\int_{G}(R_{\alpha}\ast|u|^p)|u|^p\,d\mu.
\end{eqnarray*}
By the discrete HLS inequality (\ref{bp}), one sees that $J_{\lambda}(u)\in C^{1}(E_\lambda,\R)$ for $p>\frac{N+\alpha}{N}$. Moreover, for any $u,\,v\in E_\lambda$,
\begin{eqnarray*}\label{3}
(J'_{\lambda}(u),v)=\int_{G}(\Delta u\Delta v+\nabla u\nabla v+(1+\lambda a)uv)\,d\mu-\int_{G}(R_{\alpha}\ast|u|^p)|u|^{p-2}uv \,d\mu.
\end{eqnarray*}
We say that $u\in E_\lambda$ is a solution of (\ref{1}), if $u$ is a critical point of the energy functional $J_{\lambda}$, i.e. $J'_{\lambda}(u)=0$.
A ground state solution of (\ref{1}) means that
$u$ is a nontrivial critical point of $J_\lambda$ with the least energy, that is,
$$
J_\lambda(u)=\inf\limits_{N_\lambda} J_\lambda=m_\lambda>0,
$$
where $N_{\lambda}=\{u\in E_\lambda\backslash\{0\}: (J'_\lambda(u),u)=0\}$ is the Nehari manifold related to the equation (\ref{1}).

Now we state our first result, which is about the existence of ground state solutions to the equation (\ref{1}).

\begin{thm}\label{th1}
{\rm
Assume that $N\geq 2,\,0<\alpha<N,\,p>\frac{N+\alpha}{N}$ and ($A_1$)-($A_2$) are satisfied. Then there exists $\lambda_0>0$ such that the equation (\ref{1}) has a ground state solution $u_{\lambda}\in E_\lambda$ with any $\lambda\geq\lambda_0$. Moreover, there exists a pair of nontrivial solution $(u_\lambda, v_\lambda)=(u_\lambda,R_{\alpha}\ast|u_\lambda|^p)$ of the following system
\indent
\begin{eqnarray*}
\left\{\begin{array}{ll}
\Delta^{2}u-\Delta u+(1+\lambda a(x))u=v|u|^{p-2}u, \\
(-\Delta)^{\frac{\alpha}{2}}v=|u|^{p}.
\end{array}\right.
\end{eqnarray*}
}
\end{thm}

In order to study the asymptotic behavior of $u_\lambda$ as $\lambda\rightarrow+\infty$, we introduce the limiting equation defined on the finite potential well $\Omega$
 \begin{equation}\label{2}
\left\{
\begin{array}{ll}
\Delta^{2}u-\Delta u+u=(R_{\alpha}\ast|u|^{p})|u|^{p-2}u, \qquad &\text{in}~\Omega,\\
u=0, \qquad &\text{on}~\partial\Omega.
\end{array}
\right.
\end{equation}

Different from the continuous case, the Laplacian and the gradient form of $u$ on a bounded domain $\Omega$ need additional information on the vertex boundary of $\Omega$, which is defined by $$\partial\Omega=\{y\in G\setminus\Omega:\exists~x\in\Omega~ \text{such~that}~y\sim x \}.$$
We denote $\bar{\Omega}:=\Omega\cup\partial\Omega$.

The space $W^{2,2}(\Omega)$ is defined as the set of all functions $u\in C(\bar{\Omega})$ under the norm
\begin{eqnarray*}
\|u\|_{W^{2,2}(\Omega)}=(\int_{\bar{\Omega}}(|\Delta u|^2+|\nabla u|^2)\,d\mu+\int_{\Omega}u^2\,d\mu)^{\frac{1}{2}}.
\end{eqnarray*}
It is suitable to study the equation (\ref{2}) in the space $E(\Omega):=W^{2,2}(\Omega)\cap H^{1}_0(\Omega)$, where $H^{1}_0(\Omega)$ is the completion of $C_c(\Omega)=\{u\in C(\bar{\Omega}): \text{supp}(u)\subset\Omega\}$ under the norm
\begin{eqnarray*}
\|u\|_{H^{1}_0(\Omega)}=(\int_{\bar{\Omega}}|\nabla u|^2\,d\mu+\int_{\Omega}u^2\,d\mu)^{\frac{1}{2}}.
\end{eqnarray*}
Moreover, we see that $H^{1}_0(\Omega)=C_c(\Omega)$, which differs from the continuous case.

The energy functional $J_{\Omega}(u): E(\Omega)\rightarrow\R$ related to the equation (\ref{2}) is
\begin{eqnarray*}
J_{\Omega}(u)=\frac{1}{2}\int_{\bar{\Omega}}(|\Delta u|^2+|\nabla u|^{2})\,d\mu+\frac{1}{2}\int_{\Omega}u^2\,d\mu-\frac{1}{2p}\int_{\Omega}(R_{\alpha}\ast|u|^p)|u|^p\,d\mu.
\end{eqnarray*}
Similarly, $u\in E(\Omega)$ is a ground state solution of (\ref{2}) means that
$u$ is a nontrivial critical point of $J_\Omega$ with the least energy, that is, $$J_\Omega(u)=\inf\limits_{N_\Omega} J_\Omega=m_\Omega>0,$$
where $N_{\Omega}=\{u\in E(\Omega)\backslash\{0\}: (J'_\Omega(u),u)=0\}$.



The second result is about the concentration behavior of ground state solutions as $\lambda\rightarrow+\infty$.
\begin{thm}\label{th3}
{\rm
Assume that $N\geq 2,\,0<\alpha<N,\,p>\frac{N+\alpha}{N}$ and ($A_1$)-($A_2$) are satisfied. Then for any sequence $\lambda_n\rightarrow+\infty$, up to a subsequence, the corresponding ground state solution $u_{\lambda_n}$ of (\ref{1}) converges in $W^{2,2}(G)$ to a ground state solution $u$ of (\ref{2}).
}

\end{thm}

\begin{rem}\label{4}
{\rm
\begin{itemize}

\item[(i)]
The range of the parameter $p$ for the existence of ground state solutions is $p>\frac{N+\alpha}{N}$ with $N\geq 2$, which differs from the continuous case $\frac{N+\alpha}{N}<p<\frac{N+\alpha}{N-2}$ with $N\geq 3$, thanks to the discrete nature $\ell^{s}$ into $\ell^{t}$ for any $s<t$.\\

\item[(ii)]  By the discrete HLS inequality (\ref{bz}) on $\mathbb{Z}^N$, the above results can be obtained for the Choquard equations with the Riesz potential $I_\alpha$ on $\mathbb{Z}^N$.\\

\item[(iii)]  In a locally finite graph $(\mathbb{V},\mathbb{E})$,  for the potential $V(x)=(1+\lambda a(x))$, if $a(x)\rightarrow+\infty$ as $d(x,x_0)\rightarrow+\infty$ with some $x_0\in \mathbb{V}$,   the authors in \cite{HSZ,ZZ} first proved a compact embedding $E_\lambda\looparrowright \ell^{q}(\mathbb{V})$ for $q\geq2$. Then they established the existence and convergence of ground state solutions to a class of semilinear Schr\"{o}dinger equations by the  method of Nehari manifold.  Here our assumption on $a(x)$ is weaker, by analyzing the behavior of $(PS)_c$ sequence of the functional $J_\lambda$, we get a compactness result $J_\lambda$ satisfies $(PS)_c$ condition (see Section 4).  Then we obtain similar results by the method of Nehari manifold.

\end{itemize}

}
\end{rem}

This paper is organized as follows. In Section 2, we introduce some basic results that useful in this paper. In section 3, we give some results about the Nehari manifold $N_\lambda$. In Section 4, we first prove a compactness result and then establish the existence of ground state solutions to the equation (\ref{1}). In Section 5, we show the convergence behavior of the ground state solutions to the equation (\ref{1}).

\section{Preliminaries}
 In this section, we introduce some settings for Cayley graphs, and give some useful lemmas.

Let $G$ be a countable group, it is called finitely generated if it has a finite
generating set $S$. We always assume that the generating set $S$ is symmetric, i.e.
$S=S^{-1}$. Then the Cayley graph $(G, S)$ is a locally finite graph $(V,E)$ with the set of
vertices $V = G$ and the set of edges $E=\{x\sim y: x,y\in G, x^{-1}y\in S\}$. The Cayley graph $(G, S)$ is endowed with a
natural metric, called the word metric \cite{BBI}: For any $x,y\in G$, the distance $d^{S}(x,y)$ is defined as the length of the shortest path connecting $x$ and $y$ by assigning each edge of length one, i.e.
$$d^{S}(x,y)=\inf\{k:x=x_0\sim\cdots\sim x_k=y\}.$$
One sees easily that, for any two generating sets $S_1$ and $S_2$, the metrics $d^{S_1}$ and $d^{S_2}$ are bi-Lipschitz equivalent, i.e. there exist two constants $C_1$ and $C_2$ depending on $S_1,S_2$ such that, for any $x,y\in G$,
\begin{equation}\label{jh}
C_1d^{S_2}(x,y)\leq d^{S_1}(x,y)\leq C_2d^{S_2}(x,y).
\end{equation}

Let $B^{S}_{r}(a)=\{x\in G: d^{S}(x,a)\leq r\}$ be the closed ball of radius $r$ centered at $a\in G$ and denote $|B^{S}_{r}(a)|=\sharp B^{S}_{r}(a)$ as the volume (i.e. cardinality) of the set $B^{S}_{r}(a)$. By the group structure, it is obvious that for any $a,b\in G$, $|B^{S}_{r}(a)|=|B^{S}_{r}(b)|$. The growth function of $(G, S)$ is defined as $\beta^{S}(r)=|B^{S}_{r}(e)|$, where $e$ is the unit element of $G$.

A group $G$ is called of polynomial growth, or of polynomial volume growth, if there exists a
finite generating set $S$ such that $\beta^{S}(r)\leq Cr^A$ for any $r\geq1$ and some $A >0 $. By (\ref{jh}), one sees that this definition is independent of the choice of the generating set $S$. Hence, the polynomial growth is indeed a property of the group $G$. By Gromov's theorem and Bass' volume
growth estimate of nilpotent groups \cite{B}, for any group $G$ of polynomial growth, there
are constants $C_1$ and $C_2$ depending on $S$ and $N\in \mathbb{N}$ such that, for any $r\geq 1$,
$$C_1 r^N\leq\beta^{S}(r)\leq C_2 r^N, $$
where the integer $N$ is called the homogeneous dimension or the growth degree of
$G$.

In this paper, we consider the Cayley graph $(G, S)$ of a group of polynomial
growth with the homogeneous dimension $N \geq2$. We denote by $C(\Omega)$ the space of functions on $\Omega\subset G$. For any $u\in C(\Omega)$, the $\ell^p$ norm of $u$ is defined as
 \indent
\[\ \|u\|_{\ell^p(\Omega)}=\left\{\begin{array}{ll} (\sum\limits_{x\in \Omega}|u(x)|^p)^{\frac{1}{p}},\quad &\text {if}~1\leq p<\infty, \\ \underset {x\in \Omega}{\sup}|u(x)|,\quad &\text {if}~p=\infty. \end{array}\right. \]
The $\ell^p(\Omega)$ space is defined as
$$\ell^p(\Omega)=\{u\in C(\Omega):\|u\|_{\ell^p(\Omega)}<+\infty\}.$$
We shall write $\|u\|_p$ instead of $\|u\|_{\ell^p(G)}$ if $\Omega=G$.

For $u,v\in C(G)$, the gradient form $\Gamma,$ called the ``carr\'e du champ" operator, is defined as
\begin{eqnarray*}
\Gamma(u,v)(x)=\frac{1}{2}\underset {y\sim x}{\sum}(u(y)-u(x))(v(y)-v(x))=: \nabla u \nabla v.
\end{eqnarray*}
In particular, we write $\Gamma(u)=\Gamma(u,u)$ and denote the length of $\Gamma(u)$ by
\begin{eqnarray*}
|\nabla u|(x)=\sqrt{\Gamma(u)(x)}=(\frac{1}{2}\underset {y\sim x}{\sum}(u(y)-u(x))^{2})^{\frac{1}{2}}.
\end{eqnarray*}
The Laplacian of $u$ at $x\in G$ is defined as
\begin{eqnarray*}
\Delta u(x)=\underset {y\sim x}{\sum}(u(y)-u(x)).
\end{eqnarray*}

The support of $u\in C(G)$ is
defined as $\text{supp}(u)=\{  x\in G: u(x)\neq 0 \}$. Let $C_c(G)$ be the set of all functions
with finite support. As in \cite{HSZ},  we define the biharmonic operator of $u\in C(G)$ in the distributional sense as
\begin{eqnarray*}
\int_{G}(\Delta^{2}u)\phi\,d\mu=\int_{G}\Delta u\Delta\phi\,d\mu,\quad \phi\in C_{c}(G).
\end{eqnarray*}

\begin{lm}\label{lm0}
{\rm
If $u\in E_\lambda$, then for any $2\leq q<+\infty$, $\|u\|_{q}\leq \|u\|_{E_\lambda}.$

}
\end{lm}
\begin{proof}
Clearly,  $E_\lambda\hookrightarrow \ell^{2}(G).$  In addition, for any $q\geq 2$, $\ell^{2}(G)\hookrightarrow \ell^{q}(G)$. Then one concludes that $E_\lambda\hookrightarrow \ell^{q}(G),$ which means that $\|u\|_{q}\leq \|u\|_{E_\lambda}.$
\end{proof}

By the method of subordination and Bochner's functional calculus \cite{BBK,SSV},
the fractional Laplace operator on $G$ is defined as
$$(-\Delta)^{\frac{\alpha}{2}}u=\frac{1}{|\Gamma(-\frac{\alpha}{2})|}\int_{0}^{+\infty}(e^{t\Delta}u-u)t^{-1-\frac{\alpha}{2}}\,dt,$$
where $e^{t\Delta}$ is the semigroup of $\Delta$, see \cite{KL}. And the corresponding Green's function $R_{\alpha}$ on $G$ is
$$ R_{\alpha}(x,y)=\frac{1}{|\Gamma(\frac{\alpha}{2})|}\int_{0}^{+\infty}k_t(x,y)t^{-1+\frac{\alpha}{2}}\,dt,\quad x,y\in G,$$
where $k_t(x,y)$ is the heat kernel of the Laplace operator on $G$. Since $G$ with the homogenous dimension
$N$ satisfies the weak Poincar\'{e} inequality, the heat kernel $k_t(x, y)$ has the Gaussian heat kernel bounds \cite{B1}. Hence the Green's function $R_\alpha$ of the fractional Laplacian has the asymptotic behavior $$R_{\alpha}(x,y)\simeq (d^{S}(x,y))^{\alpha-N},\quad N>\alpha.$$
Then we can prove the discrete HLS inequalities on $G$.
\begin{lm}\label{lm1}
{\rm
Let $0<\alpha <N,\,1<r,s<+\infty$ and $\frac{1}{r}+\frac{1}{s}+\frac{N-\alpha}{N}=2$. We have the discrete
HLS inequality
\begin{equation}\label{bo}
\int_{G}(R_\alpha\ast u)(x)v(x)\,d\mu\leq C_{r,s,\alpha,N}\|u\|_r\|v\|_s,\quad u\in \ell^r(G),\,v\in \ell^s(G).
\end{equation}
And an equivalent form is
\begin{equation}\label{p1}
\|R_\alpha\ast u\|_{\frac{Nr}{N-\alpha r}}\leq C_{r,\alpha,N}\|u\|_r,\quad u\in \ell^r(G),
\end{equation}
where $0<\alpha <N,\,1<r<\frac{N}{\alpha}$.
}
\end{lm}
\begin{proof}
Note that $ (d^{S}(x,y))^{\alpha-N}\in \ell^{{\frac{N}{N-\alpha},\infty}}(G)$, where the weak $\ell^{p,\infty}(G)$ space is defined as the set of all the functions $f$ on $G$ such that the quasinorm
$$\|f\|_{\ell^{p,\infty}(G)}=\inf\{C: \mu(\{x\in G:|f(x)|>s\})\leq\frac{C^{p}}{s^{p}}, \,s>0\}$$
is finite. Since the Green's function $R_\alpha$ has the asymptotic behavior $R_{\alpha}(x,y)\simeq (d^{S}(x,y))^{\alpha-N},$
by the Young's inequality for weak type spaces \cite [Theorem~1.4.25]{G}, one gets easily that
\begin{eqnarray*}
\|R_\alpha\ast u\|_{\frac{Nr}{N-\alpha r}}\leq C\|R_\alpha\|_{\ell^{\frac{N}{N-\alpha},\infty}}\|u\|_{r}\leq C\|u\|_{r}.
\end{eqnarray*}
Hence we complete the proof of the inequality (\ref{p1}). The inequality (\ref{bo}) follows from the H\"{o}lder inequality and the inequality (\ref{p1}).

\end{proof}

\begin{rem}\label{rm}
{\rm
For $u\in E_\lambda$ and $p>\frac{N+\alpha}{N}$, by taking $r=s=\frac{2N}{N+\alpha}$ in Lemma \ref{lm1}, we obtain that
\begin{equation}\label{f1}
\int_{G}(R_{\alpha}\ast|u|^p)|u|^p\, \,d\mu\leq C \|u\|_{\frac{2Np}{N+\alpha}}^{2p}\leq C\|u\|^{2p}_{E_\lambda}
\end{equation}
and
\begin{equation}\label{f2}
\|R_\alpha\ast|u|^p\|_{\frac{2N}{N-\alpha}}\leq C\|u^p\|_{\frac{2N}{N+\alpha}}.
\end{equation}

}
\end{rem}

Let $(\Omega, \Sigma, \tau)$ be a measure space, which consists of a set $\Omega$ equipped with a $\sigma-$algebra $\Sigma$ and a Borel measure $\tau:\Sigma\rightarrow[0,+\infty]$. We introduce the classical Br\'{e}zis-Lieb lemma \cite{BL}.

\begin{lm}\label{i}
{\rm(Br\'{e}zis-Lieb lemma)
Let $(\Omega, \Sigma, \tau)$ be a measure space and $\{u_n\}\subset L^{p}(\Omega, \Sigma, \tau)$ with $0<p<+\infty$. If
\begin{itemize}
\item[(a)]
$\{u_n\}$ is uniformly bounded in $L^{p}(\Omega)$,\\

\item[(b)] $u_n\rightarrow u, \tau-$almost everywhere in $\Omega$,
\end{itemize}
then we have that
\begin{eqnarray*}
\underset{n\rightarrow+\infty}{\lim}(\|u_n\|^{p}_{L^{p}(\Omega)}-\|u_n-u\|^{p}_{L^{p}(\Omega)})=\|u\|^{p}_{L^{p}(\Omega)}.
\end{eqnarray*}

}
\end{lm}

\begin{rem}\label{lm4}
{\rm
\begin{itemize}
\item[(1)]
An easy variant of the Brezis-Lieb lemma is that for $q\in[1,+\infty)$, if $\{u_n\}\subset L^{p}(\Omega)$ with $p\geq 1$ satisfies conditions (a) and (b) in Lemma \ref{i}, then for any $q\in[1,p]$,
\begin{eqnarray*}
\underset{n\rightarrow+\infty}{\lim}\int_{\Omega}||u_n|^{q}-|u_n-u|^{q}-|u|^{q}|^{\frac{p}{q}}\,d\tau=0.
\end{eqnarray*}
\\
\item[(2)] If $\Omega$ is countable and $\tau$ is the counting measure $\mu$ in $\Omega$, then we get a discrete version of the Br\'{e}zis-Lieb lemma.

\end{itemize}

}
\end{rem}



Now we prove a Br\'{e}zis-Lieb lemma for the nonlocal term $\int_{G}(R_\alpha\ast|u|^{p})|u|^{p}\, d\mu$.
\begin{lm}\label{lm3}
{\rm
Assume that $N\geq 2,\,0<\alpha<N$ and $p>\frac{N+\alpha}{N}$. Let $\{u_n\}\subset E_\lambda$ be a bounded sequence in $\ell^{\frac{2Np}{N+\alpha}}(G)$ and $u_n\rightarrow u$ pointwise in $G$. Then we have that
\begin{eqnarray*}
\underset{n\rightarrow+\infty}{\lim}(\int_{G}(R_\alpha\ast|u_n|^{p})|u_n|^{p}\,d\mu
-\int_{G}(R_\alpha\ast|u_n-u|^{p})|u_n-u|^{p}\,d\mu)=\int_{G}(R_\alpha\ast|u|^{p})|u|^{p}\,d\mu.
\end{eqnarray*}

}
\end{lm}
\begin{proof}
Direct calculation yields that
\begin{eqnarray*}
\int_{G}(R_\alpha\ast|u_n|^{p})|u_n|^{p}\,d\mu&=&\int_{G}(R_\alpha\ast|u_n-u|^{p})|u_n-u|^{p}\,d\mu\\&&+
\int_{G}(R_\alpha\ast(|u_n|^{p}-|u_n-u|^{p}))(|u_n|^{p}-|u_n-u|^{p})\,d\mu\\&& +2\int_{G}(R_\alpha\ast(|u_n|^{p}-|u_n-u|^{p}))|u_n-u|^{p} \,d\mu\\&=&\int_{G}(R_\alpha\ast|u_n-u|^{p})|u_n-u|^{p}\,d\mu+I+2II,
\end{eqnarray*}
where
$$I=\int_{G}(R_\alpha\ast(|u_n|^{p}-|u_n-u|^{p}))(|u_n|^{p}-|u_n-u|^{p})\,d\mu,\quad
II=\int_{G}(R_\alpha\ast(|u_n|^{p}-|u_n-u|^{p}))|u_n-u|^{p} \,d\mu.$$

By taking $r=\frac{2Np}{N+\alpha}$ and $q=p$ in Remark \ref{lm4} (1), we get that
\begin{equation}\label{xl}
|u_n|^{p}-|u_n-u|^{p}\rightarrow|u|^{p},\quad \text{in}~\ell^{\frac{2N}{N+\alpha}}(G).
\end{equation}

For $I$, since $\{u_n\}$ is bounded in $\ell^{\frac{2Np}{N+\alpha}}(G)$, by the HLS inequality (\ref{bo}) and (\ref{xl}), one gets that
\begin{eqnarray*}
|I- \int_{G}(R_\alpha\ast|u|^{p})|u|^{p}\, d\mu|&=&|\int_{G}(R_\alpha\ast(|u_n|^{p}-|u_n-u|^{p}-|u|^{p})(|u_n|^{p}-|u_n-u|^{p})\,d\mu\\&&+
\int_{G}(R_\alpha\ast|u|^{p})(|u_n|^{p}-|u_n-u|^{p}-|u|^{p})\,d\mu|\\&\leq&C(\||u_n|^{p}-|u_n-u|^{p}\|_{\frac{2N}{N+\alpha}}+\|u\|^{p}_{\frac{2Np}{N+\alpha}})\||u_n|^{p}-|u_n-u|^{p}-|u|^{p}\|_{\frac{2N}{N+\alpha}}\\&\rightarrow&0.
\end{eqnarray*}

For $II$, since $u_n\rightarrow u$ pointwise in $G$, passing to a subsequence if necessary, we obtain that
\begin{equation}\label{90}
|u_n-u|^{p}\rightharpoonup 0,\quad \text{in}~ \ell^{\frac{2N}{N+\alpha}}(G).
\end{equation}
Moreover, by the HLS inequality (\ref{p1}) and (\ref{xl}), one has that
\begin{equation}\label{s0}
R_\alpha\ast(|u_n|^{p}-|u_n-u|^{p})\rightarrow R_\alpha\ast|u|^{p},\qquad \text{in}~\ell^{\frac{2N}{N-\alpha}}(G).
\end{equation}
Then $II\rightarrow0$ follows from (\ref{90}) and (\ref{s0}).

\end{proof}

At the end of this section, we give a discrete Lions lemma corresponding to Lions \cite{L1} on $\mathbb{R}^{N}$, which denies a sequence $\{u_n\}$ to distribute itself over $G$.
\begin{lm}\label{lm16}
{\rm(Lions lemma)
Let $1\leq s<+\infty$. Assume that $\{u_n\}$ is bounded in $\ell^{s}(G)$ and $\|u_{n}\|_{\infty}\rightarrow0$ as $n\rightarrow+\infty.$
Then for any $s<t<+\infty$, as $n\rightarrow+\infty,$
\begin{eqnarray*}
u_n\rightarrow0,\qquad \text{in}~\ell^{t}(G).
\end{eqnarray*}}
\end{lm}
\begin{proof}
For $s<t<+\infty$, this result follows from the interpolation inequality
\begin{eqnarray*}
\|u_n\|^{t}_{t}\leq\|u_n\|_{s}^{s}\|u_n\|_{\infty}^{s-t}.
\end{eqnarray*}

\end{proof}

\section{The Neahari manifold $N_\lambda$}
 In this section, we give some results about the Nehari manifold $N_\lambda$.
 For convenience, we denote $F(u):=(J_\lambda'(u),u)$.
\begin{lm}\label{lmo}
{\rm Assume that $N\geq 2,\,0<\alpha<N$ and $p>\frac{N+\alpha}{N}$.
Then for any $u\in N_\lambda$,
\begin{itemize}
\item[(i)]
there exists $\sigma>0$ such that $\|u\|_{E_\lambda}\geq \sigma$;\\
\item[(ii)] $N_\lambda$ is a $C^{1}-$manifold. 
\end{itemize}

}
\end{lm}
\begin{proof}

(i) For $u\in N_\lambda$, by the HLS inequality (\ref{f1}), one has that
\begin{eqnarray*}
0=F(u)=\|u\|^{2}_{E_\lambda}-\int_{G}(R_{\alpha}\ast|u|^p)|u|^p\,d\mu\geq
\|u\|^{2}_{E_\lambda}-C\|u\|^{2p}_{E_\lambda}.
\end{eqnarray*}
Since $p>1$, there exists $\sigma>0$ such that $\|u\|_{E_\lambda}\geq\sigma$.\\

(ii) For $u\in N_\lambda$, by (i), we have that
\begin{eqnarray*}
(F'(u),u)&=&2\|u\|^{2}_{E_\lambda}-2p\int_{G}(R_{\alpha}\ast|u|^{p})|u|^{p}\,d\mu\\&=&
(2-2p)\|u\|^{2}_{E_\lambda}\\&\leq&(2-2p)\sigma^{2}\\&<&0.
\end{eqnarray*}
Hence $N_\lambda=\{u\in E_\lambda\setminus\{0\}: F(u)=0\}$ is a $C^{1}-$manifold by the implicit function theorem.
\end{proof}

\begin{rem}\label{lmc}
{\rm
The result (i) of Lemma \ref{lmo} implies that
\begin{eqnarray*}
m_\lambda=\inf\limits_{u\in N_\lambda} J_\lambda(u)=(\frac{1}{2}-\frac{1}{2p})\inf\limits_{u\in N_\lambda}\|u\|^{2}_{E_\lambda}\geq(\frac{1}{2}-\frac{1}{2p})\sigma^2>0.
\end{eqnarray*}

}
\end{rem}

\begin{lm}\label{lml}
{\rm Assume that $N\geq 2,\,0<\alpha<N$ and $p>\frac{N+\alpha}{N}$.
Then $u$ is a nonzero critical point of $J_\lambda$ if and only if $u$ is a critical point of $J_\lambda|_{N_\lambda}$.

}
\end{lm}
\begin{proof}
If $u$ is a nonzero critical point of $J_\lambda$, clearly  $u\in N_\lambda$ is a critical point of $J_\lambda|_{N_\lambda}$.

If $u\in N_\lambda$ is a critical point of $J_\lambda|_{N_\lambda}$, then  there exists $\eta\in\R$, a Lagrange multiplier, such that
$$J'_\lambda(u)|_{N_\lambda}=J'_\lambda(u)-\eta F'(u)=0.$$
Hence we have that
$$\eta(F'(u),u)=(J'_\lambda(u),u)=0.$$ By the result (ii) of Lemma \ref{lmo}, we get that $\eta=0$. Therefore,  $J'_\lambda(u)=0$.
\end{proof}

\begin{lm}\label{lm8}
{\rm
Assume that $N\geq 2,\,0<\alpha<N$ and $p>\frac{N+\alpha}{N}$.
For $u\in E_\lambda\setminus \{0\}$ satisfying $\|u\|^{2}_{E_\lambda}\leq\int_{G}(R_{\alpha}\ast|u|^p)|u|^{p}\,d\mu$, there exists a unique $t_u\in(0,1]$ such that $t_u u\in N_\lambda$.
}
\end{lm}

\begin{proof}
Let $u\in E_\lambda\setminus \{0)$ satisfying $0<\|u\|^{2}_{E_\lambda}\leq\int_{G}(R_{\alpha}\ast|u|^p)|u|^{p}\,d\mu$ be fixed. For $t>0$,
\begin{eqnarray*}
f(t):=J_\lambda(tu)=\frac{t^2}{2}\|u\|^{2}_{E_\lambda}-\frac{t^{2p}}{2p}\int_{G}(R_{\alpha}\ast|u|^p)|u|^{p}\,d\mu.
\end{eqnarray*}
Since $p>1$, we get that
$$f(t)\leq(\frac{t^2}{2}-\frac{t^{2p}}{2p})\int_{G}(R_{\alpha}\ast|u|^p)|u|^{p}\,d\mu\rightarrow-\infty,\quad t\rightarrow+\infty.$$
By the HLS inequality (\ref{f1}), we have that $$f(t)\geq\frac{t^2}{2}\|u\|^{2}_{E_\lambda}-\frac{t^{2p}}{2p}C\|u\|^{2p}_{E_\lambda}.$$
It is clear that $f(t)>0$ if $t>0$ is small enough.
Hence there exists $t_0>0$ such that $f(t)$ has a positive maximum and $f'(t_0)=0$. Moreover, if $tu\in N_\lambda$, then $t$ must satisfy
\begin{equation}\label{cc}
(J'_\lambda(tu),tu)=t^2\|u\|^{2}_{E_\lambda}-t^{2p}\int_{G}(R_{\alpha}\ast|u|^p)|u|^{p}\,d\mu=0.
\end{equation}
Clearly, the equation (\ref{cc}) has a unique solution $t_0=(\frac{\|u\|^{2}_{E_\lambda}}{\int_{G}(R_{\alpha}\ast|u|^p)|u|^{p}\,d\mu})^{\frac{1}{2(p-1)}}\in(0,1]$. By choosing $t_u=t_0$, we get that $t_u u\in N_\lambda$.

\end{proof}

\begin{rem}\label{f3}
{\rm
\begin{itemize}
\item[(i)] Different from the continuous case, $u\in E_\lambda\setminus \{0\}$ does not imply that $\int_{G}(R_{\alpha}\ast|u|^{p})|u|^{p}\,d\mu\neq0$. For example, $u(e)=1,\,u(x)=0$ for $x\neq e$, where $e$ is the unit element of $G$.\\

\item[(ii)] By similar arguments as in Lemma \ref{lm8}, we have that for $u\in E_\Omega\setminus \{0\}$ satisfying $\|u\|^{2}_{E_\Omega}\leq\int_{\Omega}(R_{\alpha}\ast|u|^p)|u|^{p}\,d\mu$, there exists a unique $t_u\in(0,1]$ such that $t_u u\in N_\Omega$.
\end{itemize}
}
\end{rem}

Finally, we show that the functional $J_{\lambda}(u)$ satisfies the mountain-pass geometry.
\begin{lm}\label{lm6}
{\rm
Assume that $N\geq 2,\,0<\alpha<N$ and $p>\frac{N+\alpha}{N}$. Then
\begin{itemize}
\item[(i)] there exist $\theta, \rho>0$ such that $J_{\lambda}(u)\geq\theta>0$ for $\|u\|_{E_\lambda}=\rho$;\\
\item[(ii)] there exists $e\in E_\lambda$ with $\|e\|_{E_\lambda}>\rho$ such that $J_{\lambda}(e)< 0$.
\end{itemize}
}
\end{lm}
\begin{proof}
(i) By the HLS inequality (\ref{f1}), we get that
\begin{eqnarray*}
J_{\lambda}(u)&=&\frac{1}{2}\|u\|^{2}_{E_\lambda}-\frac{1}{2p}\int_{G}(R_{\alpha}\ast|u|^p)|u|^p\,d\mu\\&\geq&
\frac{1}{2}\|u\|^{2}_{E_\lambda}-C\|u\|^{2p}_{E_\lambda}.
\end{eqnarray*}
Since $p>1$, there exist $\theta>0$ and $\rho>0$ small enough such that $J_{\lambda}(u)\geq\theta>0$ for $\|u\|_{E_\lambda}=\rho$.

\
\

(ii)
First for each $\lambda>0,\,J_{\lambda}(0)=0$. Moreover, there exists $u\in E_\lambda\setminus\{0\}$ such that $\int_{G}(R_{\alpha}\ast|u|^p)|u|^p\,d\mu\neq0$. Hence, for this fixed $u$, as $t\rightarrow+\infty$, one gets that
\begin{eqnarray*}
J_{\lambda}(tu)=\frac{t^2}{2}\|u\|^{2}_{E_\lambda}-\frac{t^{2p}}{2p}\int_{G}(R_{\alpha}\ast|u|^p)|u|^p\,d\mu\rightarrow
-\infty.
\end{eqnarray*}
Therefore, there exists $t_0>0$ large enough such that $\|t_0 u\|_{E_\lambda}>\rho$ and $J_{\lambda}(t_0 u)<0$.

\end{proof}

\section{The existence of ground state solutions}
In this section, we prove the existence of ground state solutions to the equation (\ref{1}).
Recall that, for a given functional $J_\lambda\in C^{1}(E_\lambda,\R)$, a sequence $\{u_n\}\subset E_\lambda$ is a $(PS)_c$ sequence of the functional $J_\lambda$, if it satisfies, as $n\rightarrow+\infty$,
\begin{eqnarray*}
J_\lambda(u_n)\rightarrow c, \qquad \text{in}~ E_{\lambda},\qquad\text{and}\qquad
J'_\lambda(u_n)\rightarrow 0, \qquad \text{in}~ E^{*}_{\lambda}.
\end{eqnarray*}
where $E^{*}_{\lambda}$ is the dual space of $E_{\lambda}$. Moreover, if any $(PS)_c$ sequence has a convergent subsequence, then we say that $J_\lambda$ satisfies $(PS)_c$ condition.

We first prove a compactness result, which plays a key role in the proof of Theorem \ref{th1}.
\begin{prop}\label{lm13}
{\rm
Assume that $N\geq 2,\,0<\alpha<N,\,p>\frac{N+\alpha}{N}$ and ($A_1$)-($A_2$) are satisfied. For any $c^*>0$, there exists $\lambda^*>0$ such that $J_\lambda$ satisfies $(PS)_c$ condition for all $c\leq c^*$ and $\lambda\geq\lambda^*$.

}
\end{prop}

In order to prove the Proposition \ref{lm13}, we first prove some useful lemmas about the behavior of the $(PS)_c$ sequence of the functional $J_\lambda$.
\begin{lm}\label{lm9}
{\rm
Assume that $N\geq 2,\,0<\alpha<N$ and $p>\frac{N+\alpha}{N}$. Let $\{u_n\}\subset E_\lambda$ be a $(PS)_c$ sequence of the functional $J_\lambda$. Then
\begin{itemize}
 \item[(i)] $\{u_n\}$ is bounded in $E_\lambda$;\\
\item[(ii)] $\underset{n\rightarrow+\infty}{\lim}\|u_n\|^{2}_{E_\lambda}=\frac{2pc}{p-1},$
where either $c=0$ or $c\geq c_0$ for some $c_0>0$ not depending on $\lambda$.
\end{itemize}
}
\end{lm}

\begin{proof}
Let $\{u_n\}$ be a $(PS)_c$ sequence of the functional $J_\lambda(u)$, that is,
\begin{equation}\label{1-9}
J_\lambda(u_n)\rightarrow c, \qquad \text{in}~ E_{\lambda};\qquad
J'_\lambda(u_n)\rightarrow 0, \qquad \text{in}~ E^{*}_{\lambda}.
\end{equation}

(i) It follows from (\ref{1-9}) that
\begin{eqnarray*}
(\frac{1}{2}-\frac{1}{2p})\|u_n\|^{2}_{E_\lambda}=J_\lambda(u_n)-\frac{1}{2p}(J'_\lambda(u_n),u_n)\leq C+\|J'_\lambda(u_n)\|_{ E^{*}_{\lambda}}\|u_n\|_{ E_{\lambda}},
\end{eqnarray*}
which implies that $\|u_n\|_{ E_{\lambda}}\leq C.$\\

(ii) Since $\{u_n\}$ is bounded in $E_\lambda$, we have that $\underset{n\rightarrow+\infty}{\lim}(J'_\lambda(u_n),u_n)=0$. Then
\begin{eqnarray*}
(\frac{1}{2}-\frac{1}{2p})\underset{n\rightarrow+\infty}{\lim}\|u_n\|^{2}_{E_\lambda}=\underset{n\rightarrow+\infty}{\lim}
(J_\lambda(u_n)-\frac{1}{2p}(J'_\lambda(u_n),u_n))=c.
\end{eqnarray*}
As a consequence,
\begin{equation}\label{cd}
\underset{n\rightarrow+\infty}{\lim}\|u_n\|^{2}_{E_\lambda}=\frac{2pc}{p-1}.
\end{equation}

For any $u\in E_\lambda$, by the HLS inequality (\ref{f1}), we have that
\begin{eqnarray*}\label{12}
(J'_\lambda(u),u)=\|u\|^{2}_{E_\lambda}-\int_{G}(R_{\alpha}\ast|u|^p)|u|^p\,d\mu\geq\|u\|^{2}_{E_\lambda}-C\|u\|^{2p}_{E_\lambda}.
\end{eqnarray*}
Set $\rho=(\frac{1}{2C})^{\frac{1}{2(p-1)}}$. If $\|u\|_{E_\lambda}\leq\rho$, then one gets that
\begin{equation}\label{1-3}
(J'_\lambda(u),u)\geq\frac{1}{2}\|u\|^{2}_{E_\lambda}.
\end{equation}

Let $c_0=\frac{(p-1)\rho^2}{2p}$. We claim that if $c<c_0$, then $c=0$. In fact, by (\ref{cd}), one has that
\begin{eqnarray*}
\underset{n\rightarrow+\infty}{\lim}\|u_n\|^{2}_{E_\lambda}=\frac{2pc}{p-1}<\rho^2.
\end{eqnarray*}
Hence $\|u_n\|_{E_\lambda}\leq\rho$ for large $n$. Then it follows from (\ref{1-3}) that
\begin{eqnarray*}
\frac{1}{2}\|u_n\|^{2}_{E_\lambda}\leq (J'_\lambda(u_n),u_n)\leq\|J'_\lambda(u_n)\|_{ E^{*}_{\lambda}}\|u_n\|_{E_\lambda},
\end{eqnarray*}
which implies that $\|u_n\|_{E_\lambda}\rightarrow 0$ as $n\rightarrow+\infty$. We complete the claim.

\end{proof}

\begin{lm}\label{lm10}
{\rm
Assume that $N\geq 2,\,0<\alpha<N$ and $p>\frac{N+\alpha}{N}$. Let $\{u_n\}\subset E_\lambda$ be a $(PS)_c$ sequence of the functional $J_\lambda$. Passing to a subsequence if necessary, there exists some $u\in E_\lambda$ such that
\begin{itemize}
\item[(i)] $u_n\rightharpoonup u,\qquad \text{in}~E_{\lambda}$;\\
\item[(ii)] $u_n\rightarrow u,\qquad \text{pointwise~in}~G$;\\
\item[(iii)] $J'_\lambda(u)=0,\qquad \text{in}~E^{*}_{\lambda}$.
\end{itemize}
}
\end{lm}
\begin{proof}
(i)
By Lemma \ref{lm9}, $\{u_n\}$ is bounded in $E_\lambda$.
Then passing to a subsequence if necessary, there exists some $u\in E_\lambda$ such that $u_n\rightharpoonup u$ as $n\rightarrow+\infty$.

\
\

(ii) Clearly, $\{u_n\}\subset E_{\lambda}$ is bounded in $\ell^2(G)$, and hence bounded in $\ell^{\infty}(G)$. Therefore, by diagonal principle, there exists a subsequence of $\{u_n\}$  pointwise converging to $u$.

\
\

(iii) It is sufficient to show that for any $v\in C_{c}(G)$, $(J'_\lambda(u),v)=0.$
For any $v\in C_{c}(G)$, assume that $\text{supp}(v)\subseteq B^{S}_{r}(e)$, where $r$ is a positive constant. Since $B^{S}_{r+1}(e)\subset G$ is a finite set and $u_n\rightarrow u $ pointwise in $G$ as $n\rightarrow+\infty$, for any $s\geq 1$,
\begin{equation}\label{84}
u_n\rightarrow u,\qquad \text{in}~\ell^{s}(B^{S}_{r+1}(e)).
\end{equation}
We claim that $|(J'_\lambda(u_n)-J'_\lambda(u),v)|\rightarrow0$ as $n\rightarrow+\infty$. In fact, direct calculation yields that
\begin{eqnarray*}
|(J'_\lambda(u_n)-J'_\lambda(u),v)|&\leq&
\int_{B^S_{r+1}(e)}(|\Delta(u_n-u)||\Delta v|+|\nabla(u_n-u)||\nabla v|)\,d\mu\\&&+\int_{B^S_{r}(e)}(1+\lambda a)|u_n-u||v| \,d\mu\\&&+|\int_{B^S_{r}(e)}[(R_\alpha\ast|u_n|^p)|u_n|^{p-2}u_n-(R_\alpha\ast|u|^p)|u|^{p-2}u]v \,d\mu|
\\&=:& I_1+I_2+I_3.
\end{eqnarray*}

For $I_1$ and $I_2$, by (\ref{84}), one gets easily that $I_{1}\rightarrow0,\,I_2\rightarrow0$ as $n\rightarrow+\infty.$ For $I_3$,
\begin{eqnarray*}
I_3 &=& |\int_{B^S_{r}(e)}[(R_\alpha\ast|u_n|^p)|u_n|^{p-2}u_n-(R_\alpha\ast|u|^p)|u|^{p-2}u]v\,d\mu|\\&\leq&|\int_{B^S_{r}(e)}(R_\alpha\ast|u_n|^p)|u|^{p-2}uv \,d\mu-\int_{B^S_{r}(e)}(R_\alpha\ast|u|^p)|u|^{p-2}uv\,d\mu|\\&&+
|\int_{B^S_{r}(e)}(R_\alpha\ast|u_n|^p)(|u_n|^{p-2}u_n-|u|^{p-2}u)v \,d\mu|
\\&=:& J_1+J_2.
\end{eqnarray*}
We first prove that $J_1\rightarrow 0$ as $n\rightarrow+\infty$. By the boundedness of $\{u_n\}\subset E_\lambda$ in $\ell^{\frac{2Np}{N+\alpha}}(G)$ and the HLS inequality (\ref{f2}), we get that $\{R_\alpha\ast|u_n|^p\}$ is bounded in $\ell^{\frac{2N}{N-\alpha}}(G)$. Thus, passing to a subsequence if necessary, we have that
$R_\alpha\ast|u_n|^p\rightharpoonup R_\alpha\ast|u|^p$ in $\ell^{\frac{2N}{N-\alpha}}(G)$. Note that $|u|^{p-2}uv\in \ell^{\frac{2N}{N+\alpha}}(G)$, then $J_1\rightarrow0$ as $n\rightarrow+\infty$.

Next, we prove that $J_2\rightarrow 0$ as $n\rightarrow+\infty$. By the HLS inequality (\ref{bo}), the boundedness of $\{u_n\}\subset E_{\lambda}$ in $\ell^{\frac{2Np}{N+\alpha}}(G)$ and (\ref{84}), one gets that
\begin{eqnarray*}
J_2&=&|\int_{B^S_{r}(e)}(R_\alpha\ast|u_n|^p)(|u_n|^{p-2}u_n-|u|^{p-2}u)v \,d\mu|\\&\leq&
C\|u_n\|^p_\frac{2Np}{N+\alpha}(\int_{B^S_{r}(e)}||u_n|^{p-2}u_n-|u|^{p-2}u|
^{\frac{2N}{N+\alpha}}|v|^{\frac{2N}{N+\alpha}}\,d\mu)^{\frac{N+\alpha}{2N}}\\&\leq& C(\int_{B^S_{r}(e)}||u_n|^{p-2}u_n-|u|^{p-2}u|
^{\frac{2Np}{(N+\alpha)(p-1)}}\,d\mu)^{\frac{(N+\alpha)(p-1)}{2Np}}(\int_{B_r}|v|^{\frac{2Np}{N+\alpha}}\,d\mu)^{\frac{N+\alpha}{2Np}}\\&\leq&
C\||u_n|^{p-2}u_n-|u|^{p-2}u\|_{\ell^{\frac{2Np}{(N+\alpha)(p-1)}}(B^S_{r}(e))}\\&\rightarrow&0, \qquad n\rightarrow+\infty.
\end{eqnarray*}
Therefore, we have that $I_3\rightarrow0$ as $n\rightarrow+\infty$. In a summary, up to a subsequence, we prove that $(J'_\lambda(u_n)-J'_\lambda(u),v)\rightarrow0$ as $n\rightarrow+\infty$. Then $(J'_\lambda(u),v)=0$ follows from $(J'_\lambda(u_n),v)\rightarrow0$ as $n\rightarrow+\infty$.

\end{proof}

\begin{lm}\label{lm11}
{\rm
Assume that $N\geq 2,\,0<\alpha<N$ and $p>\frac{N+\alpha}{N}$. Let $\{u_n\}\subset E_\lambda$ be a $(PS)_c$ sequence of the functional $J_\lambda$. Passing to a subsequence if necessary, there exists some $u\in E_\lambda$ such that
\begin{itemize}
\item[(i)] $\underset{n\rightarrow+\infty}{\lim}J_\lambda(u_n-u)=c-J_\lambda(u)$;\\
\item[(ii)] $\underset{n\rightarrow+\infty}{\lim}J'_\lambda(u_n-u)=0,\qquad \text{in~}E^{*}_\lambda$.
\end{itemize}
}
\end{lm}

\begin{proof}
It follows from Lemma \ref{lm10} that
\begin{eqnarray*}
\begin{array}{ll}
\|u_n\|_{E_\lambda}\leq C\qquad \text{and}\qquad
u_n\rightarrow u, \qquad \text{pointwise~in}~G.
\end{array}
\end{eqnarray*}

(i) By the Br\'{e}zis-Lieb lemma, we have that
\begin{equation}\label{bb}
\int_{G}(1+\lambda a)|u_n|^2\,d\mu-\int_{G}(1+\lambda a)|u_n-u|^2\,d\mu=\int_{G}(1+\lambda a)|u|^2 \,d\mu+o(1).
\end{equation}

Similar to the proof of Lemma 3.3 in \cite{HLW}, one has that
\begin{equation}\label{bc}
\int_{G}|\nabla u_n|^2\,d\mu-\int_{G}|\nabla (u_n-u)|^2 \,d\mu=\int_{G}|\nabla u|^2 \,d\mu+o(1).
\end{equation}

In addition, since $\{\Delta u_n\}$ is bounded in $\ell^2(G)$ and $\Delta u_n\rightarrow\Delta u$ pointwise in $G$ as $n\rightarrow+\infty$, by the Br\'{e}zis-Lieb lemma, we obtain that
\begin{equation}\label{bd}
\int_{G}|\Delta u_n|^2\,d\mu-\int_{G}|\Delta (u_n-u)|^2 \,d\mu=\int_{G}|\Delta u|^2\,d\mu+o(1).
\end{equation}

Then by (\ref{bb}), (\ref{bc}), (\ref{bd}) and Lemma \ref{lm3}, one gets that
\begin{eqnarray*}
J_\lambda(u_n)-J_\lambda(u_n-u)=J_\lambda(u)+o(1).
\end{eqnarray*}
Note that $\underset{n\rightarrow+\infty}{\lim}J_\lambda(u_n)=c$, we obtain the desired result
\begin{eqnarray*}
J_\lambda(u_n-u)=c-J_\lambda(u)+o(1).
\end{eqnarray*}

(ii)
For any $v\in C_{c}(G)$, assume that $\text{supp}(v)\subseteq B^{S}_{r}(e)$, where $r$ is a positive constant. By the HLS inequality (\ref{bo}), the boundedness of $\{u_n\}$ in $E_\lambda$ and the H\"{o}lder inequality, one has that
\begin{eqnarray*}
\int_{B^{S}_{r}(e)}(R_\alpha\ast|u_n-u|^p)|u_n-u|^{p-1}|v| \,d\mu&\leq&
C\|u_n-u\|^{p}_{\frac{2Np}{N+\alpha}}(\int_{B^{S}_{r}(e)}|u_n-u|
^{\frac{2N(p-1)}{N+\alpha}}|v|^{\frac{2N}{N+\alpha}}\,d\mu)^{\frac{N+\alpha}{2N}}\\&\leq& C\|u_n-u\|^{p}_{ E_\lambda}(\int_{B^{S}_{r}(e)}|u_n-u|
^{\frac{2N(p-1)}{N+\alpha}}|v|^{\frac{2N}{N+\alpha}}\,d\mu)^{\frac{N+\alpha}{2N}}\\&\leq& C\|u_n-u\|^{(p-1)}_{\ell^{\frac{2Np}{N+\alpha}}(B^{S}_{r}(e))}\|v\|_{\frac{2Np}{N+\alpha}}\\&\leq&
C\|u_n-u\|^{(p-1)}_{\ell^{2}(B^{S}_{r}(e))}\|v\|_{E_\lambda}.
\end{eqnarray*}
Since $B^S_{r+1}(e)$ is a finite set and $u_n\rightarrow u $ pointwise in $G$ as $n\rightarrow+\infty$, we get that
\begin{eqnarray*}
|(J'(u_n-u),v)|&\leq& \int_{B^S_{r+1}(e)}(|\Delta(u_n-u)||\Delta v|+|\nabla(u_n-u)||\nabla v|)\,d\mu+\int_{B^S_{r}(e)}(1+\lambda a)|u_n-u||v| \,d\mu\\&&+\int_{B^{S}_{r}(e)}(R_\alpha\ast|u_n-u|^p)|u_n-u|^{p-1}|v| \,d\mu\\&\leq&
\|\Delta(u_n-u)\|_{\ell^{2}(B^S_{r+1}(e))}\|\Delta v\|_{2}+\|\nabla(u_n-u)\|_{\ell^{2}(B^S_{r+1}(e))}\|\nabla v\|_{2}\\&&+\|(1+\lambda a)^{\frac{1}{2}}(u_n-u)\|_{\ell^{2}(B^S_{r}(e))}
\|(1+\lambda a)^{\frac{1}{2}}v\|_{2}+C\|u_n-u\|^{(p-1)}_{\ell^{2}(B^S_{r}(e))}\|v\|_{E_\lambda}\\&\leq&C\xi_{n}\|v\|_{E_\lambda},
\end{eqnarray*}
where $C$ is a constant not depending on $n$ and $\xi_{n}\rightarrow 0$ as $n\rightarrow+\infty$.
Therefore, we get that
$$\underset{n\rightarrow+\infty}{\lim}\|J'(u_n-u)\|_{E^{*}_\lambda}=\underset{n\rightarrow+\infty}{\lim}
\underset{\|v\|_{E_\lambda}=1}{\sup}|(J'(u_n-u),v)|=0.$$

\end{proof}

\begin{lm}\label{lm12}
{\rm
Assume that $N\geq 2,\,0<\alpha<N,\,p>\frac{N+\alpha}{N}$ and $(A_1)-(A_2)$ are satisfied. Let $c_*$ be a fixed constant. For any $\varepsilon>0$, there exist $\lambda_{\varepsilon}>0$ and $r_{\varepsilon}>0$ such that if $\{u_n\}\subset E_\lambda$ is a $(PS)_c$ sequence of the functional $J_\lambda$ with $c\leq c_*$ and $\lambda\geq\lambda_{\varepsilon}$, then we have that
\begin{eqnarray*}
\underset{n\rightarrow+\infty}{\limsup}\int_{G\setminus B^{S}_{r_{\varepsilon}}(e)}(R_{\alpha}\ast|u_n|^p)|u_n|^p \,d\mu\leq\varepsilon.
\end{eqnarray*}
}
\end{lm}

\begin{proof}
For $r\geq1$, let
\begin{eqnarray*}
\Omega^{+}_{r}=\{x\in G: d^{S}(x,e)> r,~ a(x)\geq M\}\qquad \text{and}\qquad \Omega^{-}_{r}=\{x\in G: d^{S}(x,e)> r,~ a(x)< M\}.
\end{eqnarray*}
 By  Lemma \ref{lm9}, for $n$ large enough, one has that
\begin{eqnarray*}
\int_{\Omega^{+}_{r}}|u_n|^2\,d\mu&\leq&\frac{1}{1+\lambda M}\int_{G}(1+\lambda a(x))|u_n|^2\,d\mu\\&\leq&
\frac{1}{1+\lambda M}\|u_n\|^2_{E_\lambda}\\&\leq&\frac{1}{1+\lambda M}(\frac{2pc}{p-1}+1)\\&\leq&\frac{1}{1+\lambda M}(\frac{2pc_*}{p-1}+1)
\\&\rightarrow&0, \qquad\lambda\rightarrow+\infty.
\end{eqnarray*}

For $q>1$, by the H\"{o}lder inequality and  the condition ($A_2$), for $n$ large enough, we get that
\begin{eqnarray*}
\int_{\Omega^{-}_{r}}|u_n|^2\,d\mu&\leq&(\int_{G}|u_n|^{2q}\,d\mu)^{\frac{1}{q}}(\mu(\Omega^{-}_{r}))^{1-\frac{1}{q}}\\
&\leq&\|u_n\|^2_{E_\lambda}(\mu(\Omega^{-}_{r}))^{1-\frac{1}{q}}\\&\leq&(\frac{2pc_*}{p-1}+1)(\mu(\Omega^{-}_{r}))^{1-\frac{1}{q}}\\&\rightarrow&0, \qquad r\rightarrow+\infty.
\end{eqnarray*}

Let $\phi\in C(G)$ such that $\phi(x)=1$ for $d^{S}(x,e)\geq r+1$ and $\phi(x)=0$ for $d^{S}(x,e)\leq r$. By Lemma \ref{lm1} and the fact that $\|u\|_{\ell^\frac{2Np}{N+\alpha}(G\setminus B^{S}_r(e))}\leq \|u\|_{\ell^2(G\setminus B^{S}_r(e))}$ with $\frac{2Np}{N+\alpha}\geq2$, for $n$ large enough, we have that
\begin{eqnarray*}
\int_{G\setminus B^{S}_r(e)}(R_{\alpha}\ast|u_n|^p)|u_n|^p \,d\mu&=&\int_{G}(R_{\alpha}\ast|u_n|^p)|\phi u_n|^p \,d\mu\\&\leq&C(\int_{G}|u_n|^{\frac{2Np}{N+\alpha}}\,d\mu)^{\frac{N+\alpha}{2N}}(\int_{G}|\phi u_n|^{\frac{2Np}{N+\alpha}}\,d\mu)^{\frac{N+\alpha}{2N}}\\&\leq&C\|u_n\|^{p}_{E_\lambda}(\int_{G\setminus B^{S}_r(e)}|u_n|^{\frac{2Np}{N+\alpha}}\,d\mu)^{\frac{N+\alpha}{2N}}
\\&\leq&C(\frac{2p c_*}{p-1}+1)^{p}\|u_n\|^{p}_{\ell^{\frac{2Np}{N+\alpha}}(G\setminus B^{S}_r(e))}
\\&\leq&C\|u_n\|^{p}_{\ell^{2}(G\setminus B^{S}_r(e))}\\&=&C(\int_{\Omega^{+}_{r}}|u_n|^2 \,d\mu+\int_{\Omega^{-}_{r}}|u_n|^2 \,d\mu)^{\frac{p}{2}}
\\&\rightarrow&0, \qquad\lambda,~r\rightarrow+\infty.
\end{eqnarray*}

\end{proof}
Based on the arguments discussed above, now we prove the Proposition \ref{lm13}.

\
\

{\bf Proof of Proposition \ref{lm13}}:
Let $0<\varepsilon<\frac{pc_0}{p-1}$, where $c_0$ is the constant in Lemma \ref{lm9}. Then for the given $c^*>0$, we choose $\lambda^*=\lambda_{\varepsilon}>0$ and $r_\varepsilon>0$ as in Lemma \ref{lm12}.

Let $\{u_n\}\subset E_\lambda$ be a $(PS)_c$ sequence of the functional $J_\lambda$ with $c\leq c^*$ and $\lambda\geq\lambda^*$. By Lemma \ref{lm10}, there exists $u\in E_\lambda$ such that
\begin{eqnarray*}
u_n\rightharpoonup u,\qquad \text{in}~E_\lambda,\qquad\text{and}\qquad u_n\rightarrow u,\qquad \text{pointwise~in}~G.
\end{eqnarray*}
Denote $$v_n(x)=u_n(x)-u(x).$$
By Lemma \ref{lm11}, one sees that $\{v_n\}\subset E_\lambda$ is a $(PS)_d$ sequence of the functional $J_\lambda$ with $d=c-J_\lambda(u)$, i.e.
\begin{eqnarray*}
J_\lambda(v_n)\rightarrow d, \qquad \text{in}~ E_{\lambda},\qquad\text{and}\qquad
J'_\lambda(v_n)\rightarrow 0, \qquad \text{in}~ E^{*}_{\lambda}.
\end{eqnarray*}
We claim that $d=0$. By contradiction, if $d\neq0$, then by Lemma \ref{lm9}, $d\geq c_0>0$. Direct calculation yields that
\begin{eqnarray*}
d+o(1)&=&J_\lambda(v_n)-\frac{1}{2}(J'_\lambda(v_n),v_n)\\&=&(\frac{1}{2}-\frac{1}{2p})\int_{G}(R_{\alpha}\ast|v_n|^p)|v_n|^p\,d\mu,
\end{eqnarray*}
which means that
\begin{eqnarray*}
\underset{n\rightarrow+\infty}{\lim}\int_{G}(R_{\alpha}\ast|v_n|^p)|v_n|^p\,d\mu=\frac{2pd}{p-1}\geq\frac{2pc_0}{p-1}.
\end{eqnarray*}
On the other hand, by Lemma \ref{lm12}, we have that
\begin{eqnarray*}
\underset{n\rightarrow+\infty}{\limsup}\int_{G\setminus B^{S}_{r_{\varepsilon}}(e)}(R_{\alpha}\ast|v_n|^p)|v_n|^p \,d\mu\leq\varepsilon<\frac{pc_0}{p-1}.
\end{eqnarray*}
The above two inequalities imply that $v_n\rightarrow v$ pintwise in $G$ with some $v\neq 0$, which contradicts $v_n\rightarrow 0$ pointwise in $G$. Hence $d=0$. By Lemma \ref{lm9}, we get that
\begin{eqnarray*}
\underset{n\rightarrow+\infty}{\lim}\|v_n\|^{2}_{E_\lambda}=\frac{2pd}{p-1}=0.
\end{eqnarray*}
That is $u_n\rightarrow u$ in $E_\lambda$.\qed

In the following, we prove the existence of ground state solutions  to the equation (\ref{1}).

\
\

{\bf Proof of Theorem \ref{th1}}: By Lemma \ref{lm6}, one sees that $J_\lambda$ satisfies the mountain-pass geometry. Hence there exists a sequence $\{u_n\}\subset E_\lambda$ such that
\begin{eqnarray*}\label{3-9}
J_\lambda(u_n)\rightarrow m_\lambda, \qquad\text{in~}E_\lambda,\qquad\text{and}\qquad J'_\lambda(u_n)\rightarrow 0,\qquad\text{in~} E^{*}_\lambda.
\end{eqnarray*}
Then it follows from Lemma \ref{lm10} that there exists $u_0\in E_\lambda$ such that up to a subsequence,
\begin{eqnarray*}\label{3-0}
\begin{array}{ll}
u_n\rightharpoonup u_{0},\qquad \text{in}~E_\lambda,\\
u_n\rightarrow u_{0}, \qquad \text{pointwise~in}~G,\\
J'_{\lambda}(u_{0})=0,\qquad \text{in}~E^{*}_\lambda.
\end{array}
\end{eqnarray*}
By proposition \ref{lm13}, there exists $\lambda_0>0$ such that, for any $\lambda\geq\lambda_0$, $u_n\rightarrow u_0$ in $E_\lambda$.
Then it follows from Lemma \ref{lm9} and $m_\lambda>0$ that
\begin{eqnarray*}
\|u_0\|^{2}_{E_\lambda}=\underset{n\rightarrow+\infty}{\lim}\|u_n\|^{2}_{E_\lambda}=\frac{2pm_\lambda}{p-1}>0,
\end{eqnarray*}
which yields $u_0\neq0$. By Lemma \ref{lml}, one sees that $u_0\in N_\lambda$. Then
\begin{eqnarray*}
J_\lambda(u_0)&=&J_\lambda(u_0)-\frac{1}{2p}(J'_\lambda(u_0),u_0)\\&=&(\frac{1}{2}-\frac{1}{2p})\|u_0\|^{2}_{E_\lambda}\\&=&
\underset{n\rightarrow+\infty}{\lim}~(\frac{1}{2}-\frac{1}{2p})\|u_n\|^{2}_{E_\lambda}\\&=&\underset{n\rightarrow+\infty}{\lim}~
(J_\lambda(u_n)-\frac{1}{2p}(J'_\lambda(u_n),u_n))\\&=&m_\lambda>0.
\end{eqnarray*}
Therefore, $u_\lambda=u_0\in N_\lambda$ is a ground state solution of (\ref{1}).\qed

\section{The convergence of ground state solutions}

In this section, we show that as $\lambda\rightarrow+\infty$, the ground state solution $u_\lambda$ of (\ref{1}) converges to a ground state solution of (\ref{2}) in $W^{2,2}(G)$. First, we prove the asymptotic behavior of $m_\lambda$ as $\lambda\rightarrow+\infty.$
\begin{lm}\label{lm14}
{\rm
Assume that $N\geq 2,\,0<\alpha<N,\,p>\frac{N+\alpha}{N}$ and ($A_1$)-($A_2$) are satisfied. Then
$m_\lambda\rightarrow m_\Omega$ as $\lambda\rightarrow+\infty$.
}
\end{lm}

\begin{proof}
First, for any $\lambda>0$, since $N_\Omega\subset N_\lambda$, we have that $m_\lambda\leq m_\Omega$ .
We prove this lemma by contradiction.
Assume that there exists a sequence $\lambda_n\rightarrow+\infty$ such that $\underset{n\rightarrow+\infty}{\lim}m_{\lambda_n}=k<m_\Omega$. By Theorem \ref{th1}, there exists a sequence $\{u_{\lambda_n}\}\subset N_{\lambda_n}$, ground state solutions of (\ref{1}), such that $J_{\lambda_n}(u_{\lambda_n})=m_{\lambda_n}>0$. Since $\{u_{\lambda_n}\}$ is uniformly bounded in $W^{2,2}(G)$, up to a subsequence, there exists $u\in W^{2,2}(G)$ such that
\begin{eqnarray*}
u_{\lambda_n}\rightharpoonup u,\qquad \text{in}~W^{2,2}(G),\quad
u_{\lambda_n}\rightarrow u, \qquad \text{pointwise~in}~G.
\end{eqnarray*}

Now we claim that $u=0$ in $\Omega^{c}$, where $\Omega^{c}=G\setminus\Omega$. In fact, if $u\neq0$ in $\Omega^{c}$, then there exists a vertex $x_0\not\in\Omega$ such that $u(x_0)\neq 0$. Clearly, $a(x_0)>0, u_{\lambda_n}(x_0)\rightarrow u(x_0)\neq0$. Since $u_{\lambda_n}\in N_{\lambda_n}$ and $\lambda_n\rightarrow+\infty$, we get that
\begin{eqnarray*}
J_{\lambda_n}(u_{\lambda_n})&=&(\frac{1}{2}-\frac{1}{2p})\|u_{\lambda_n}\|^2_{E_{\lambda_n}}\\&\geq&(\frac{1}{2}-\frac{1}{2p})\int_{G}\lambda_n a(x) |u_{\lambda_n}|^2 \,d\mu\\&\geq&(\frac{1}{2}-\frac{1}{2p})\lambda_n a(x_0)|u_{\lambda_n}(x_0)|^2\\&\rightarrow&+\infty,\qquad n\rightarrow+\infty,
\end{eqnarray*}
which contradicts the fact that $k< m_\Omega$. Hence $u=0$ in $\Omega^{c}$.

Next we prove that $u_{\lambda_n}\rightarrow u$ in $\ell^{q}(G)$ for any $q>2$. Otherwise, by Lemma \ref{lm16}, there exists $\delta>0$ such that $\underset{n\rightarrow+\infty}{\lim}\|u_{\lambda_n}-u\|_{\infty}=\delta>0$. Then there exists a sequence $\{x_{n}\}\subset G$ such that $|(u_{\lambda_n}-u)(x_{n})|\geq\frac{\delta}{2}>0.$
Since $(u_{\lambda_n}-u)(x)\rightarrow 0$ pointwise in $G$ as $n\rightarrow+\infty$, we have that $d^S(x_n,e)\rightarrow+\infty$ as $n\rightarrow+\infty$.

Note that $u_{\lambda_n}\in N_{\lambda_n}$ and $\mu(B^{S}_{r}(x_n)\cap\{x: a(x)\leq M\})\rightarrow 0$ as $n\rightarrow+\infty$, where $0<r<1$. Then we have that
\begin{eqnarray*}
J_{\lambda_n}(u_{\lambda_n})&\geq&(\frac{1}{2}-\frac{1}{2p})\int_{B^{S}_{r}(x_n)\cap\{x: a(x)\geq M\}}\lambda_n a(x) u^2_{\lambda_n} \,d\mu\\&=&
(\frac{1}{2}-\frac{1}{2p})\int_{B^{S}_{r}(x_n)\cap\{x: a(x)\geq M\}}\lambda_n a(x) |u_{\lambda_n}-u|^2 \,d\mu\\&\geq&(\frac{1}{2}-\frac{1}{2p})\lambda_n M(\int_{B^{S}_{r}(x_n)} |u_{\lambda_n}-u|^2 \,d\mu-\int_{B^{S}_{r}(x_n)\cap\{x: a(x)\leq M\}}|u_{\lambda_n}-u|^2 \,d\mu)\\&=&
(\frac{1}{2}-\frac{1}{2p})\lambda_n M(|u_{\lambda_n}-u|^{2}(x_{n})-\int_{B^{S}_{r}(x_n)\cap\{x: a(x)\leq M\}}|u_{\lambda_n}-u|^2 \,d\mu)\\&\geq&
(\frac{1}{2}-\frac{1}{2p})\lambda_n M(\frac{\delta^2}{4}+o(1))\\&\rightarrow&+\infty,\qquad n\rightarrow+\infty.
\end{eqnarray*}
This is a contradiction. Hence for any $q>2$, $u_{\lambda_n}\rightarrow u$ in $\ell^{q}(G).$
Since $\frac{2Np}{N+\alpha}>2$, by the HLS inequality (\ref{f1}), we get that
\begin{eqnarray*}
\int_{G}(R_{\alpha}\ast|u_{\lambda_n}-u|^p)|u_{\lambda_n}-u|^p \,d\mu\leq C\|u_{\lambda_n}-u\|^{2p}_\frac{2Np}{N+\alpha}\rightarrow0,\qquad n\rightarrow+\infty.
\end{eqnarray*}
Then by Lemma \ref{lm3}, we get that
\begin{equation}\label{31}
\underset {n\rightarrow+\infty}{\lim}\int_{G}(R_{\alpha}\ast|u_{\lambda_n}|^p)|u_{\lambda_n}|^p \,d\mu=\int_{G}(R_{\alpha}\ast|u|^p)|u|^p \,d\mu.
\end{equation}
This implies that $u\neq0$ in $G$. In fact, if $u=0$ in $G$, note that $u_{\lambda_n}\in N_{\lambda_n}$, by the result (i) of Lemma \ref{lmo} and (\ref{31}), we obtain that
\begin{eqnarray*}
0=\underset {n\rightarrow+\infty}{\lim}\int_{G}(R_{\alpha}\ast|u_{\lambda_n}|^p)|u_{\lambda_n}|^p \,d\mu=\underset {n\rightarrow+\infty}{\lim}
\|u_{\lambda_n}\|^{2}_{E_{\lambda_n}}\geq\sigma^2>0,
\end{eqnarray*}
which is a contradiction. Hence $u\neq0$ in $\Omega$. By the weaker lower semi-continuity of the norm $\|\cdot\|_{W^{2,2}}$, $u=0$ in $\Omega^{c}$ and (\ref{31}), one has that
\begin{eqnarray*}
\int_{\bar{\Omega}}(|\Delta u|^2+|\nabla u|^2)\,d\mu+\int_{\Omega}|u|^2\,d\mu&\leq&\int_{G}(|\Delta u|^2+|\nabla u|^2+|u|^2)\,d\mu\\&\leq&\underset {n\rightarrow+\infty}{\liminf}~\int_{G}(|\Delta u_{\lambda_n}|^2+|\nabla u_{\lambda_n}|^2+|u_{\lambda_n}|^2)\,d\mu\\&\leq&\underset {n\rightarrow+\infty}{\liminf}~\int_{G}(|\Delta u_{\lambda_n}|^2+|\nabla u_{\lambda_n}|^2+(1+\lambda_n a)|u_{\lambda_n}|^2)\,d\mu\\&=&\underset {n\rightarrow+\infty}{\liminf}~\int_{G}(R_{\alpha}\ast|u_{\lambda_n}|^p)|u_{\lambda_n}|^p \,d\mu\\&=&\int_{\Omega}(R_{\alpha}\ast|u|^p)|u|^p \,d\mu.
\end{eqnarray*}
By the result (ii) of Remark \ref{f3}, there exists a unique $t\in(0,1]$ such that $tu\in N_\Omega$. Then
\begin{eqnarray*}
J_{\Omega}(tu)&=&(\frac{1}{2}-\frac{1}{2p})(\int_{\bar{\Omega}}(|t\Delta u|^2+|t\nabla u|^2)\,d\mu+\int_{\Omega}|tu|^2\,d\mu)\\&\leq&(\frac{1}{2}-\frac{1}{2p})\int_{G}(|t\Delta u|^2+|t\nabla u|^2+|tu|^2)\,d\mu\\&\leq&
\underset {n\rightarrow+\infty}{\liminf}~(\frac{1}{2}-\frac{1}{2p})\int_{G}(|t\Delta u_{\lambda_n}|^2+|t\nabla u_{\lambda_n}|^2+(1+\lambda_n a)|tu_{\lambda_n}|^2)\,d\mu\\&=&t^2~\underset {n\rightarrow+\infty}{\liminf}~J_{\lambda_n}(u_{\lambda_n})\\&\leq&
\underset {n\rightarrow+\infty}{\liminf}~m_{\lambda_n}\\&=&k.
\end{eqnarray*}
Consequently, we have that $$m_\Omega\leq J_{\Omega}(tu)\leq k<m_\Omega.$$ This is a contradiction.

\end{proof}

\
\

{\bf Proof of Theorem \ref{th3}}: We prove that for any sequence $\lambda_n\rightarrow+\infty$,  up to a subsequence, the corresponding ground state solution $u_{\lambda_n}\in N_{\lambda_n}$ converges to a ground state solution $u$ of (\ref{2}) in $W^{2,2}(G)$. Since $\{u_{\lambda_n}\}$ is uniformly bounded in $W^{2,2}(G)$, there exists $u\in W^{2,2}(G)$ such that
\begin{eqnarray*}
u_{\lambda_n}\rightharpoonup u,\qquad \text{in}~W^{2,2}(G),\quad
u_{\lambda_n}\rightarrow u, \qquad \text{pointwise~in}~G.
\end{eqnarray*}
By similar arguments as in Lemma \ref{lm14}, we can show that $u=0$ in $\Omega^{c}$ and $u_{\lambda_n}\rightarrow u$ in $\ell^{q}(G)$ with $q>2$,
and thus
\begin{equation}\label{df}
\underset {n\rightarrow+\infty}{\lim}\int_{G}(R_{\alpha}\ast|u_{\lambda_n}|^p)|u_{\lambda_n}|^p \,d\mu=\int_{G}(R_{\alpha}\ast|u|^p)|u|^p \,d\mu.
\end{equation}
This means that $u\neq0$ in $\Omega$. Now we prove that $u\in E(\Omega)$ is a ground state solution of $(\ref{2})$. In fact, since $J'_{\lambda_n}(u_{\lambda_n})=0$, for any $v\in C_c(\Omega)$,
\begin{eqnarray*}
\int_{\bar{\Omega}}(\Delta u_{\lambda_n}\Delta v+\nabla u_{\lambda_n}\nabla v)\,d\mu+\int_{\Omega}(1+\lambda_n a)u_{\lambda_n}v\,d\mu=\int_{\Omega}(R_{\alpha}\ast|u_{\lambda_n}|^p)|u_{\lambda_n}|^{p-2}u_{\lambda_n}v \,d\mu.
\end{eqnarray*}
Note that $a(x)=0$ for $x\in \Omega$. Moreover, since $\bar{\Omega}$ is a finite set and $u_{\lambda_n}\rightarrow u$ pointwise in $G$, similar to the proof of the result (iii) in Lemma \ref{lm10}, one gets that
\begin{eqnarray*}
\int_{\bar{\Omega}}(\Delta u\Delta v+\nabla u\nabla v)+\int_{\Omega}uv\,d\mu=\int_{\Omega}(R_{\alpha}\ast|u|^p)|u|^{p-2}uv \,d\mu.
\end{eqnarray*}
Hence $u$ is a nonzero critical point for the functional $J_{\Omega}$, i.e. $J'_{\Omega}(u)=0$. By Lemma \ref{lml}, one has that $u\in N_\Omega$.

By $u=0$ in $\Omega^{c}$ and (\ref{df}), one gets that
\begin{eqnarray*}
m_{\lambda_n}=J_{\lambda_n}(u_{\lambda_n})&=&(\frac{1}{2}-\frac{1}{2p})\int_{G}(R_{\alpha}\ast|u_{\lambda_n}|^p)|u_{\lambda_n}|^p \,d\mu\\&=&(\frac{1}{2}-\frac{1}{2p})\int_{G}(R_{\alpha}\ast|u|^p)|u|^p \,d\mu +o(1)\\&=&(\frac{1}{2}-\frac{1}{2p})\int_{\Omega}(R_{\alpha}\ast|u|^p)|u|^p \,d\mu+o(1)\\&=&J_{\Omega}(u)+o(1).
\end{eqnarray*}
By Lemma \ref{lm14}, we get that $J_{\Omega}(u)=m_\Omega$. Hence, $u$ is a ground state solution of (\ref{2}).

Note that $u_{\lambda_n}\in N_{\lambda_n}$, $u\in N_{\Omega}$ with $u=0$ in $\Omega^{c}$, by the Br\'{e}zis-Lieb lemma and (\ref{df}), one concludes that
\begin{eqnarray*}
\|u_{\lambda_n}-u\|^{2}_{E_{\lambda_n}}&=&\int_{G}|\Delta(u_{\lambda_n}-u)|^2+|\nabla(u_{\lambda_n}-u)|^2+(1+\lambda_n a)|u_{\lambda_n}-u|^2\,d\mu\\&=&
\|u_{\lambda_n}\|^{2}_{E_{\lambda_n}}-\|u\|^{2}_{E(\Omega)}+o(1)\\&=&
\int_{G}(R_{\alpha}\ast|u_{\lambda_n}|^p)|u_{\lambda_n}|^p \,d\mu-\int_{\Omega}(R_{\alpha}\ast|u|^p)|u|^p \,d\mu+o(1)\\&=&o(1).
\end{eqnarray*}
Hence $u_{\lambda_n}\rightarrow u$ in $W^{2,2}(G)$. The proof is completed.

\
\

\section{Acknowledgements}
We would like to take this opportunity to express our gratitude to Bobo Hua, who has given us so much valuable suggestions on our results, and has tried his best to improve our paper. 

\end{document}